\documentclass[11pt,reqno]{amsart}
\usepackage{amsmath}
\usepackage{cases}
\usepackage{mathrsfs}
\usepackage{bbm}
\usepackage{amssymb}
\usepackage{amscd}
\usepackage{amsfonts,latexsym,amsmath,
amsthm,amsxtra,mathdots,amssymb,latexsym,mathabx}
\usepackage[all,cmtip]{xy}
\RequirePackage{amsmath} \RequirePackage{amssymb}
\usepackage{color}
\usepackage{colordvi}
\usepackage{multicol}
\usepackage{hyperref}
\usepackage{mathtools}
\usepackage[margin=1in]{geometry}
\usepackage{xcolor}

\hypersetup{
    colorlinks,
    linkcolor={red!50!black},
    citecolor={blue!100!black},
    urlcolor={blue!100!black}
}
\usepackage{cite}

\newcommand{\bea}{\begin{eqnarray}}
\newcommand{\eea}{\end{eqnarray}}
\newcommand{\bna}{\begin{eqnarray*}}
\newcommand{\ena}{\end{eqnarray*}}

\numberwithin{equation}{section}

\theoremstyle{plain}
\newtheorem{lemma}{Lemma}[section]
\newtheorem{theorem}[lemma]{Theorem}
\newtheorem{corollary}[lemma]{Corollary}
\newtheorem{proposition}[lemma]{Proposition}

\theoremstyle{definition}
\newtheorem{definition}[lemma]{Definition}
\newtheorem{remark}{Remark}

\renewcommand{\Re}{\operatorname{Re}}
\renewcommand{\Im}{\operatorname{Im}}

\newcommand{\SL}{\operatorname{SL}}

\newcommand{\supp}{\operatorname{supp}}
\newcommand{\stab}{\operatorname{stab}}
\newcommand{\tr}{\operatorname{tr}}

\newcommand{\Hbb}{\mathbb{H}}

\newcommand{\Zbb}{\mathbb{Z}}
\newcommand{\Fcal}{\mathcal{F}}

\title[]
{Equidistribution of holomorphic cusp forms on thin sets}

\author{Qingfeng Sun}
	\address{School of Mathematics and Statistics, Shandong University, Weihai\\Weihai, Shandong 264209, China}
	\address{State Key Laboratory of Cryptography and Digital Economy
Security \\ Shandong University, Weihai, 264209, China}
	\email{qfsun@sdu.edu.cn}

    \author{Qizhi Zhang}
	\address{School of Mathematics and Statistics, Shandong University, Weihai\\Weihai, Shandong 264209, China}
    \email{qzzhang@mail.sdu.edu.cn}



\thanks{Q. Sun was partially supported by the National Natural Science Foundation of China (Grant Nos.
12471005 and 12031008) and
the Natural Science Foundation of Shandong Province (Grant No. ZR2023MA003).}

\date{}

\begin{document}
\begin{abstract}
We find some equidistribution results connected to restriction quantum unique ergodicity problem in this paper. We shows that
\begin{align*}
  \frac{1}{|\mathcal{B}_k|}\sum_{f\in \mathcal{B}_k} \int_{R}y^{k}|f(z)|^{2}\psi(z) d\mu_{R}(z)\to \frac{3}{\pi}\int_{R}\psi(z) d\mu_{R}(z)
\end{align*}
where $R$ is some subset of $\mathbb{H}$, $\psi$ is a nice function relative to $R$,  $d\mu_{R}(z)$ is a suitable measure on $R$, and $\mathcal{B}_k$ is an orthonormal basis of the cusp forms for group $\Gamma$ with respect to weight $k$.
\end{abstract}
\maketitle

\section{Introduction}

The Quantum Unique Ergodicity (QUE) conjecture is now a celebrated theorem,
established in the work of Lindenstrauss, Holowinsky and Soundararajan
\cite{lindenstraussInvariantMeasuresArithmetic2006, soundararajanQuantumUniqueErgodicity2010, holowinskyMassEquidistributionHecke2010c}. Nevertheless, numerous related problems remain of great interest, such as the QUE problem on shrinking sets,
where one investigates the smallest scale at which equidistribution persists.
The focus of this paper is the rQUE (restrict QUE) conjecture formulated
by Young\cite{youngQuantumUniqueErgodicity2016}, which predicts that
the sequence of measures induced by Hecke cusp forms converges
in the weak-* topology on certain specific submanifolds.
For instance, on the vertical geodesic $(\mathbb{R}^{+}, y^{-1}dy)$
and the unit interval $(\mathbb{R}/\mathbb{Z},dx)$, the respective limit relations are:
\[
y^k|f_{j}(z)|^2\frac{dy}{y}=d\mu_f^{\times} \xrightarrow{w*} \frac{3}{\pi} \frac{dy}{y}
\]
and for some $\Im z>0$,
\[
y^k|f_{j}(z)|^2dx=d\mu_f^{+} \xrightarrow{w*} \frac{3}{\pi} dx,
\]
where in both cases $f_j$ runs over all $L^2$-normalized holomorphic Hecke cusp forms
as the weight $k$ tends to infinity.

Establishing this result unconditionally is a hard task.
As Young mentions in \cite{youngQuantumUniqueErgodicity2016}, an optimal bound of $k^{\varepsilon}$
 for the integral is known to imply the subconvexity estimate
 for the associated $L$-function at the central point:
\[
L(1/2, f) \ll C(f)^{1/8 + \varepsilon},
\]
a bound whose strength remains unproven even for the Riemann zeta function.
While the unconditional problem remains open, Zenz proved that the result holds
on average by analyzing the quantum variance of rQUE\cite{Zenz2021QuantumVF}
(see also the authors \cite{qfsunqzzhang}).

In this paper, we derive the limiting distribution associated with the
rQUE problem by averaging over an orthonormal basis of cusp forms,
and prove that it conforms to the conjectured behavior.
The following notations will be used throughout the paper:

\textbf{Notation}.
\begin{itemize}
\item Let $\Gamma=\SL(2,\mathbb{Z})$ and
$X=\Gamma\backslash \mathbb{H}$ is the fundamental domain;
\item $a,b,c,d$ will typically denote the entries of a $2 \times 2$ matrix in $\Gamma$;
\item Let $\mathcal{B}_k=\{f_{j}\}$ be an orthonormal basis of
the space $S(\Gamma, k)$ of holomorphic cusp forms of weight $k$. Here $k$ will be sufficiently large;
\item Let
\bea
d\mu^{\times}_{k}&=& |\mathcal{B}_k|^{-1}
\sum_{f\in \mathcal{B}_k}y^k|f(x+iy)|^2\frac{dy}{y},\label{1.1}\\
d\mu^{+}_{k}&=& |\mathcal{B}_k|^{-1}\sum_{f\in \mathcal{B}_k}y^k|f(x+iy)|^2dx,\label{1.2}\\
d\mu^{X}_{k}&=& |\mathcal{B}_k|^{-1}
\sum_{f\in \mathcal{B}_k}y^k|f(z)|^2\frac{dxdy}{y^2};\label{1.3}
\eea
\item Let $\psi$ be a smooth, compactly supported function defined on either $\mathbb{R}^+$ or $\mathbb{R}/\mathbb{Z}$. The specific choice of domain will be clear from the context;
\item $C, C_1,\cdots$ denote some absolute positive constants, and $A$ denotes a fixed sufficiently large number;
\item $\delta$ denotes some small positive number, which may depend on $k$. $\varepsilon$ is a small but fixed absolute positive number.
\end{itemize}

Luo\cite{Luo03} showed that for any measurable subset $A$ on the modular surface $X$ and any $\varepsilon>0$,
\begin{align*}
  \int_{A}d\mu^{X}_{k}\to \frac{3}{\pi}\int_{A}\frac{dxdy}{y^2}.
\end{align*}
The main observation of this paper is that the average over an
orthonormal basis of $S(\Gamma, k)$ can be converted into a summation
over the group $\Gamma$. The identity (see Definition \ref{Bargmannkernel}
and Proposition \ref{innerproduct}) can be viewed as the pre-trace formula.
Motivated by Luo's work, our approach bypasses the need for any analysis of
$L$-functions, relying instead on a
carefully analysis
 of
the action of the group $\Gamma$ on the upper half-plane. We investigate
the average result on certain submanifolds and derive
the following asymptotic formula:

\begin{theorem}\label{mainresult1}
  Let $d\mu^{\times}_{k}$ be defined as in \eqref{1.1} and
  let $Y$ be sufficiently large. For $\mathop{supp}\psi\subset (Y^{-1}, k^{\frac{1}{2}}(17A\log k)^{-\frac{1}{2}})$, we have
  \begin{align*}
    \int_{\mathbb{R}^+}\psi(y)d\mu_{k}^{\times}= \frac{3}{\pi} \int_{\mathbb{R}^+} \psi(y)\frac{dy}{y}+O(||\psi||_{\infty}Y^2 k^{-\frac{1}{2}}(\log k)^{\frac{1}{2}}),
  \end{align*}
 where the implied constant is absolute.
\end{theorem}
\begin{theorem}\label{mainresult2}
Let $d\mu^{+}_{k}$ be defined as in \eqref{1.2} and
let $Y$ be sufficiently large. For $y\in (Y^{-1}, k^{\frac{1}{2}}(17A\log k)^{-\frac{1}{2}})$, we have
\begin{align*}
    \int_{\mathbb{R}/\mathbb{Z}} \psi(x)d\mu_{k}^{+}= \frac{3}{\pi} \int_{\mathbb{R}/\mathbb{Z}} \psi(x) dx+O(||\psi||_{\infty}Y^2 k^{-\frac{1}{2}}(\log k)^{\frac{1}{2}}),
  \end{align*}
 where the implied constant is absolute, and $\psi(x)$ is only required to be integrable.
\end{theorem}
\begin{remark}
  In fact, for the assumption of Lemma \ref{ellipticdistance} to hold, it suffices that $Y$ be greater than the maximum order of the elliptic points of $\Gamma$.
\end{remark}
By choosing $\phi(x,y)$ as a smooth, compactly-supported function
with $supp(\phi)\subset
\SL_2(\mathbb{Z})\setminus\mathbb{H}$, we apply
Theorem \ref{mainresult1} for fixed $x\in \mathbb{R}$ and $y^{-1}\phi(x,y)$.
By performing integration on both sides, we obtain the following result,
which is an extension of the work by Luo\cite{Luo03}.
\begin{corollary}
Let $d\mu^{X}_{k}$ be defined as in \eqref{1.3}.
	Suppose that $\phi$ is a smooth, compactly-supported function with
$supp(\psi)\subset \SL_2(\mathbb{Z})\setminus\mathbb{H}$. We have
	\begin{align*}
	\int_{X} \phi(z)d\mu_{k}^{X}= \frac{3}{\pi} \int_{X} \phi(z) \frac{dxdy}{y^2}+O_{\phi}( k^{-\frac{1}{2}}(\log k)^{\frac{1}{2}}),
	\end{align*}
for $k$ sufficiently large.
\end{corollary}

\section{The Bergman Kernel and Its Asymptotic Properties}

In this section, we introduce the definition of Bergman kernel.
We refine the results of Huang, Lester, Wigman, and Yesha\cite{HLWY25},
which is important to our subsequent estimates.

\subsection{Definition and Relation to the Covariance Kernel}

Let $k \ge 4$ be an even integer.
For $z, w \in \Hbb^2$ and $\gamma =
\begin{pmatrix} a & b \\ c & d \end{pmatrix} \in \Gamma$, we define
\[
b_\gamma(z, w) := \frac{2i}{w + \gamma z} \frac{1}{cz+d}.
\]
The Bergman kernel for weight $k$ cusp forms on $\Gamma$ is then defined as the sum over the modular group.

\begin{definition}[Bergman Kernel]\label{Bargmannkernel}
For $z, w \in \Hbb^2$, the Bergman kernel is given by
\[
B_k(z, w) := \sum_{\gamma \in \Gamma} b_\gamma(z, w)^k.
\]
\end{definition}
The key to studying the rQUE problem in an average sense lies in the
following properties of the Bargmann kernel\cite[Theorem 2.15]{SteinerUniformbounds}.
\begin{proposition}\label{innerproduct}
  We have
  \begin{align*}
    \sum_{f\in\mathcal{B}_k} f(z)\overline{f(w)}= \frac{k-1}{8\pi}B(z,-\overline{w})
  \end{align*}
  for any $z,w\in\mathbb{H}$.
\end{proposition}

It's clear to see that the function $B_k(\cdot, \cdot)$ is
holomorphic in both variables and serves as a reproducing kernel
for the space of weight $k$ cusp forms. For our purposes, it is
more convenient to work with a normalized version of the Bergman kernel,
denoted by $R_k(z, w)$, which is defined as
\begin{equation} \label{eq:bergman_Rk_def}
R_k(z, w) := (yv)^{k/2} B_k(z, -\bar{w}),
\end{equation}
where $z=x+iy$ and $w=u+iv$.

\subsection{Asymptotic Behavior in the Bulk and Near Elliptic Points}

To state the asymptotic result, we first define the relevant subdomains. Let $Y>0$ and $P(Y)$ be
\begin{align*}
  P(Y)=\big\{z\in \Hbb\big| |\Re z|\leq\frac{1}{2}, \Im z>Y^{-1}\big\}.
\end{align*}
We have by Lemma 2.10 in \cite{Iwa95}, which states that there are
roughly $Y$ copies of the standard fundamental domain in that region.
Hence there are $3Y$ elliptic points in $P(Y)$ at most, where the elliptic points
satisfy that the stability group
\begin{align*}
  \stab(z)=\{\gamma\in\Gamma|\gamma z=z\}
\end{align*}
has non-trivial elements, i.e., the classical fundamental domain has three
elliptic points $i$, $e^{\pi i/3}$, and $e^{2\pi i/3}$. We then
denote $E(Y)$, which is a finite set, as the collection of all the elliptic points in $P(Y)$. For a given $\delta > 0$, we define the $\delta$-neighborhoods of these points as
\begin{align}\label{ellipticpoint}
\eta_{\delta, j} := \{ z \in P(Y) : d_{\Hbb}(z, e_j) \le \delta \}, \quad e_j\in E(Y).
\end{align}
The "bulk" of the fundamental domain, which excludes these neighborhoods,
is denoted by $\Fcal_\delta$:
\[
\Fcal_\delta := P(Y) \setminus (\bigcup_{j=1}^{|E(Y)|}\eta_{\delta,j}).
\]

\begin{lemma}
  Let $\delta>0$ be a suitable small number. Let $\gamma=\begin{pmatrix} a & b \\ c & d \end{pmatrix}\in \Gamma\setminus\{\pm I\}$ with $\tr \gamma<2$. Denote the fixed point of $\gamma$ as $z_0$, where $z_0\in \Hbb$.  Then for any $z\in\mathbb{H}$, if $d_{\Hbb}(z,\gamma z)>\delta$, we have
  \[d_{\Hbb}(z,z_0)> \frac{\delta}{2}.\]
\end{lemma}
\begin{proof}
  By Proposition 2.1 in \cite{topics}, the stability group is finite
  cyclic for such $z_0$. And we further suppose that $m$ is the least
  positive integer such that $\gamma^m=I$. The hyperbolic metric is invariant under the group action in $\SL(2,\mathbb{R})$, which implies
 \begin{align*}
   d_{\Hbb}(z,\gamma z)=d_{\Hbb}(\gamma^j z,\gamma^{j+1} z)\quad\text{and }\quad d_{\Hbb}(z,z_0)=d_{\Hbb}(\gamma^j z,z_0),\quad \text{for }j \in\mathbb{Z}.
 \end{align*}
 We assert that the points listed below are mutually distinct for $z\neq z_0$, otherwise $z$ is a fixed point of some $\gamma^{j}\neq \pm I$, but the only fixed point of such $\gamma^{j}$ is $z_0$, which is a contradiction,
 \begin{align*}
   z,\gamma z, \gamma^2 z,\dots,\gamma^{m-2}z, \gamma^{m-1}z.
 \end{align*}
 We prove that
 \begin{align*}
   d_{\Hbb}(z,z_0)=\frac{1}{m}\sum_{j=0}^{m-1} d_{\Hbb}(\gamma^j z,z_0)=\frac{1}{2m}\sum_{j=0}^{m-1} ( d_{\Hbb}(\gamma^j z,z_0)+d_{\Hbb}(\gamma^{j+1} z,z_0))> \frac{1}{2m} m \delta=\frac{\delta}{2}.
 \end{align*}
\end{proof}

We now show that for any non-elliptic element $\gamma \in \Gamma\setminus\{\pm I\}$ (i.e., satisfying $\tr(\gamma) \ge 2$), the hyperbolic distance $d(z, \gamma z)$ has a uniform positive lower bound for all $z$ in the region $\{ z \in \mathbb{H} : Y^{-1} < \mathrm{Im}(z) < 2 \}$. This bound depends only on $Y$.
The hyperbolic distance
 $d_{\Hbb}(z,w)$ is related to the formula $\cosh d_{\Hbb}(z,w)=2u(z,w)+1$, where
 \begin{align*}
   u(z,w)=\frac{|z-w|^2}{4\Im(z)\Im(w)}.
 \end{align*}
Let $\gamma=\begin{pmatrix} a & b \\ c & d \end{pmatrix}\in
\Gamma\setminus\{\pm I\}$ with $\tr \gamma\geq 2$. A direct computation shows that
\begin{align*}
  u(z,\gamma z)=(2y)^{-2}|cz^2+(d-a)z-b|^2.
\end{align*}
When $c=0$, we must have that $a=d=\pm 1$ and $d\neq 0$, i.e., $\gamma\in\Gamma_{\infty}$, which implies for $z\in \{ z \in \mathbb{H} : Y^{-1} < \mathrm{Im}(z) < 2 \},$
 $$
 u(z,\gamma z)=(2y)^{-2}b^2\geq \frac{1}{16}.
 $$
When $c \neq 0$, the fixed points of $\gamma$ are the roots of the quadratic equation $cz^2 + (d-a)z - b = 0$, which are given by:
\begin{align*}
  z_{\pm}=\frac{a-d\pm\sqrt{(a+d)^2-4}}{2c}.
\end{align*}
When $z\in \{z\in\Hbb|Y^{-1}<\Im z<2\}$, then the distance $d_{\Hbb}(z,\gamma z)$ is
\begin{align*}
  u(z,\gamma z)=(2y)^{-2}c^2|(z-z_+)(z-z_-)|^2\geq (2y)^2c^2.
\end{align*}
The lower bound $d_{\Hbb}(z,\gamma z)\geq (2y)^2c^2$ is given by the fact that all the zeros $z_{\pm}$ are real number in this case.
We have established that if $\tr \gamma\geq 2$ and $\gamma\neq\pm I$, then
\begin{align}\label{lowerboundforzgz}
d_{\Hbb}(z,\gamma z)\geq \min\{(2Y)^{-1},2^{-1}\},
\end{align}
when $z$ in the horizontal strip $\{z\in\Hbb \,|\, Y^{-1}<\Im z<2\}$.

The elliptic case, where $\tr(\gamma) < 2$, requires a more delicate
analysis. Since the fixed point $z_0$ is in the upper half-plane,
the action of $\gamma$ is a hyperbolic rotation. While this
implies that $z$ and $\gamma(z)$ share a hyperbolic circle
centered at $z_0$, we will show that the discrete nature of
the stabilizer subgroups in $\Gamma$ precludes the possibility
of their distance becoming arbitrarily small.
\begin{lemma}\label{ellipticdistance}
  Let $1>\delta>0$. Let $Y$ exceed the order of any stabilizer subgroup of an elliptic point for $\Gamma$. For any $\gamma\in\Gamma\setminus\{\pm I\}$, we have the uniform lower bound for $z\in \mathcal{F}_{\delta}\cap \{z\in\Hbb \,|\, \Im z<2\}$,
  \[d_{\Hbb}(z,\gamma z)>\frac{\delta}{4Y}.\]
\end{lemma}

\begin{proof}
We first give a state of Hyperbolic Law of Cosines: A hyperbolic triangle with side lengths $\mathbf{a}$, $\mathbf{b}$, and $\mathbf{c}$. Let $\rho$ be the angle at the vertex opposite the side of length $\mathbf{c}$. The relationship is given by the formula:
$\cosh(\mathbf{c}) = \cosh(\mathbf{a})\cosh(\mathbf{b}) - \sinh(\mathbf{a})\sinh(\mathbf{b})\cos(\rho),$
where $\mathbf{a}$, $\mathbf{b}$, and $\mathbf{c}$ are the hyperbolic lengths of the sides.

Since the action of $\SL(2, \mathbb{R})$ on $\mathbb{H}$ is transitive, all points lie in the same orbit. It is a fundamental principle that stabilizers of points within the same orbit are conjugate. Therefore, given a point $z$ with an $2m$-th ($m\leq2$ by noting that we work in $\SL(2, \mathbb{R})$) order stabilizer $\stab_{\Gamma}(z)$, we may choose $\gamma_z \in \SL(2, \mathbb{R})$ with $\gamma_z(i) = z$ and study the conjugate group $\gamma_z^{-1} \stab_{\Gamma}(z) \gamma_z$, which is $2m$-th order stabilizer of $i$. Moreover, an order-$2m$ stabilizer subgroup of $i$ is a cyclic group generated by the matrix $k(\frac{\pi}{m})$, given by:
\[
k(\theta) = \begin{pmatrix} \cos(\theta/2) & \sin(\theta/2) \\ -\sin(\theta/2) & \cos(\theta/2) \end{pmatrix}, \quad \text{with } \theta = \frac{\pi}{m}.
\]
The angle subtended at the point $i$ by the geodesic radii to a point $z$ and its image $k(\theta)z$ is equal to $\theta$. Indeed, the angle of rotation is recovered by computing the derivative of the transformation at its fixed point $i$.

Now we can show there is a lower bound for $d(z,\gamma z)$, where $z\in \mathcal{F}_{\delta}\cap \{z\in\Hbb \,|\, \Im z<2\}$. The choice of $\delta$ and $Y$ allows us only consider the case $\tr \gamma<2$.

Suppose there exists a $\gamma\in \Gamma$ with $\tr \gamma<2$, such that there exists $z_1\in \mathcal{F}_{\delta}\cap \{z\in\Hbb \,|\, \Im z<2\}$,
\begin{align*}
  d_{\mathbb{H}}(z_1,\gamma z_1)<\frac{\delta}{4Y}.
\end{align*}
Suppose the fixed point of $\gamma$ is $z_0$ with a $m$-order stabilizer and denote $r=d_{\mathbb{H}}(z_0,\gamma z_1)=d_{\mathbb{H}}(z_0,z_1)>\delta$.
Then we have a hyperbolic triangle with side lengths $r$, $r$, and $d_{\mathbb{H}}(z_1,\gamma z_1)$, and $\theta$ be the angle at the vertex opposite the side of length $d_{\mathbb{H}}(z_1,\gamma z_1)$. Applying the Hyperbolic Law of Cosines, we have that
\begin{align*}
  \cosh(d_{\mathbb{H}}(z_1,\gamma z_1)) = \cosh(r)^2 - \sinh(r)^2\cos(\theta).
\end{align*}
The distance $d_{\mathbb{H}}(z_1, \gamma z_1)$ increases monotonically with the angle $\theta \in (0, \pi)$. As the minimal angle is constrained by the $2m$-th order stabilizer to be $\pi/m$, we obtain the inequality:
\begin{align*}
  \cosh(d_{\mathbb{H}}(z_1,\gamma z_1))=1+ \sinh(r)^2\sin^2(\frac{\pi}{2m})>1+ \sinh(r)^2 m^{-2}.
\end{align*}
The restriction of $r>\delta$ implies that
$\sinh(r)^2\sin^2(\frac{\pi}{2m})>\left( m^{-1} \delta\right)^2$. Hence the following inequality holds
\begin{align*}
  \cosh(d_{\mathbb{H}}(z_1,\gamma z_1))>1+ \left(m^{-1} \delta \right)^2,
\end{align*}
which is a contradiction with $d_{\mathbb{H}}(z_1,\gamma z_1)<\frac{\delta}{4Y}$.
We have finished the proof.
\end{proof}

The following theorem provides a precise asymptotic estimate for the Bergman kernel within this bulk region\cite[Theorem 3.3]{HLWY25}.

\begin{theorem} \label{thm:bergman_asymptotic}
Let $k \ge 4$ be even. Let $z=x+iy$ and $w=u+iv$. Under the assumptions of Theorem \ref{ellipticdistance}, there exists
some constant $c_0>0$ satisfying $d_{\Hbb}(z,w)<c_0\delta$, such that  we have a uniform estimate in the region $\mathcal{F}_{\delta}$.
\[
R_k(z, w) = 2 \left( \frac{2i\sqrt{yv}}{z-\bar{w}} \right)^k + O\left(  e^{-\frac{\delta^2}{128Y^2}k} + y e^{-k/(17y^2)} \right).
\]
\end{theorem}
\begin{proof}
We start with the estimate in the proof of Theorem 3.3 in \cite{HLWY25}:
\[ R_k(z, w) = 2 \left( \frac{2i\sqrt{yv}}{z-\bar{w}} \right)^k + \mathcal{E}_k(z, w), \]
where the error term can be estimated as
\begin{align*}
   |\mathcal{E}_k(z, w)| &\le \sum_{\gamma \in \SL_2(\Zbb), \gamma \neq \pm I} \left(1 + u(w, \gamma z)\right)^{-k/2}\\
   &\le \Im(z)\max_{\gamma\in \SL(2,\mathbb{Z})\atop \gamma\neq\pm I}
   (1+u(w,\gamma z))^{-\frac{k}{2}+2}.
\end{align*}
By (3.14) and (3.15) in \cite{HLWY25}, when $\Im z >2$,  we have
\begin{align*}
  R_k(z, w)=O(ye^{-k/(17y^2)}).
\end{align*}

Lemma \ref{ellipticdistance} shows that the action of $\Gamma$ on the domain $\{z\in \mathbb{H}\big||\Re z|\leq
\frac{1}{2}, Y^{-1}<\Im z\leq 2\}\cap \Fcal_\delta$ satisfies
\begin{align*}
  d_{\Hbb}(z,\gamma z)\geq \frac{\delta}{4Y}.
\end{align*}
Since $c_0$ is sufficiently small, we have
\begin{align*}
  d_{\Hbb}(w,\gamma z)\geq d_{\Hbb}(z,\gamma z)-d_{\Hbb}(w,z)\geq \frac{\delta}{8Y}.
\end{align*}
By the relation $\cosh d_{\Hbb}(w,\gamma z)=2u(w,\gamma z)+1$, we obtain the inequality
\begin{align*}
  u(w,\gamma z)=\frac{\cosh d_{\Hbb}(w,\gamma z)-1}{2}>\frac{\delta^2}{128 Y^2}.
\end{align*}
Then we conclude that in domain $\{z\in \mathbb{H}\big||\Re z|\leq
\frac{1}{2}, Y^{-1}<\Im z\leq 2\}\cap \Fcal_\delta$,
\begin{align*}
   |\mathcal{E}_k(z, w)|  &\le \Im(z)\max_{\gamma\in \SL(2,\mathbb{Z})\atop \gamma\neq\pm I}(1+\frac{\delta^2}{128Y^2})^{-\frac{k}{2}+2}\\
   &\ll  e^{-\frac{\delta^2}{128 Y^2}k}.
\end{align*}
\end{proof}

Near the elliptic points, additional terms arising from the non-trivial
stabilizer groups become significant. We state the following as a remark.

\begin{remark}[Refinement near the elliptic point $z_0$]
Since the hyperbolic metric is invariant under the action of $\Gamma$, it follows that the set of elliptic points is discrete. Consequently, the hyperbolic distance between any two distinct elliptic points has a positive infimum. It follows from the assumptions and proof of Theorem \ref{thm:bergman_asymptotic} that for any elliptic point, it is possible to choose a  suitable $\delta$-neighborhood wherein the main term of the Bergman kernel arises solely from the stabilizer subgroup of that point.
For $z$ in a neighborhood of the elliptic point $z_0$ (e.g., $|z-z_0| \le \delta$) and $|z-w| \le c_0 \delta$, the asymptotic formula for the Bergman kernel includes an additional main term:
\begin{align}\label{neari}
R_k(z, w) = (\Im z\Im w)^{\frac{k}{2}}\sum_{\gamma\in \stab(z_0)}b_\gamma(z, -\overline{w})^k + O\left( e^{-\frac{\delta^2}{128Y^2}k} \right),
\end{align}
where we shall note that $\stab(z_0)$ is a finite subgroup of $\SL(2,\mathbb{Z})$.
\end{remark}

The Bergman kernel provides sufficient information to study
the rQUE problem in an average sense. Let $z=w=x+iy$,
Theorem \ref{thm:bergman_asymptotic} shows that for
all $0 < \delta < 1$, $z \in \Fcal_\delta$, and $d_{\Hbb}(w,z) \le c_0 \delta$, we have
\begin{align}\label{faraway}
R_k(z, z) = 2 + O\left( e^{-\frac{\delta^2}{128Y^2}k} + y e^{-k/(17y^2)} \right).
\end{align}
By the remark of Theorem \ref{thm:bergman_asymptotic}, for any
small $\delta$-neighborhood $\{z|d_{\Hbb}(z,z_0)<\delta\}$ of an
elliptic point $z_0$ in $P(Y)$, we have
\begin{align}\label{ellipticnearasymptotic}
R_k(z, z) = 2+ \sum_{\gamma\in \stab(z_0)\setminus\{\pm I\}}(y b_\gamma(z, -\overline{z}))^k + O\left( e^{-\frac{\delta^2}{128Y^2}k} \right).
\end{align}

The following equation holds
\begin{align*}
  |y b_\gamma(z, -\overline{z})|=\frac{2y}{|\gamma z-\overline{z}|}
  =\frac{2y}{(|\gamma z-z|^2+4y^2)^{\frac{1}{2}}}=(1+u(z,\gamma z))^{-\frac{1}{2}}\leq(\frac{1}{2}\cosh d_{\Hbb}(z,\gamma z))^{-\frac{1}{2}}.
\end{align*}
Let $\delta=\frac{\sqrt{128A} Y(\log k)^{\frac{1}{2}}}{ k^{\frac{1}{2}}}$.
By Lemma \ref{ellipticdistance}, we get an upper bound of the
non-trivial action,
\begin{align}\label{ellipticnearerrorterm}
  |y b_\gamma(z, -\overline{z})|^{k}\ll (1+d_{\Hbb}^2(z,\gamma z))^{-\frac{k}{2}}\ll e^{-A(\log k)}.
\end{align}
The choice of $\delta$ also implies the error term in \eqref{ellipticnearasymptotic} is bounded by
\begin{align}\label{esimate1}
 e^{-\frac{\delta^2}{128Y^2}k}\ll e^{-A(\log k)}
\end{align}
When $d_{\Hbb}(z,z_0)< \delta$, we use the trivial bound
\begin{align}\label{ellipticnearnearerrorterm}
  |y b_\gamma(z, -\overline{z})|^{k}\leq 1.
\end{align}

\section{Equidistribution results via Bergman Kernel}
 Let $\psi$ be a compactly supported smooth function supported
 on $\mathbb{R}^{+}$. Let $z=x+iy$ and $|x|\leq 1/2$. We investigate
 \begin{align*}
   \int_{\Re(z)=x\atop y>0}R(z,z)\psi(y)\frac{dy}{y}.
 \end{align*}

 Recall the notation \eqref{ellipticpoint}. Take $\delta=\frac{\sqrt{128A}Y(\log k)^{\frac{1}{2}}}{ k^{\frac{1}{2}}}$.  If there is no elliptic point near the geodesic line $\{\Re(z)=x, y>0\}$, i.e.,
 \begin{align*}
   \big(\bigcup_{j}\eta_{\delta, j}\big)\;\bigcap \;\{|\Re(z)|\leq \frac{1}{2}, \Im(z)>0\}=\varnothing.
 \end{align*}
 Noting that $\supp(\psi)\subset(Y^{-1}, \frac{k^{\frac{1}{2}}}{\sqrt{17A}(\log k)^{\frac{1}{2}}})$,  we directly  deduce by \eqref{faraway} and \eqref{esimate1},
 \begin{align}\label{farawayavr}
   \int_{\Re(z)=x\atop y>0}R(z,z)\psi(y)\frac{dy}{y}= 2\int_{0}^{\infty}\psi(y)\frac{dy}{y}+ O\left( k^{-A}\right).
 \end{align}
  By \eqref{ellipticnearasymptotic}, \eqref{ellipticnearerrorterm}
   and \eqref{ellipticnearnearerrorterm}, if there is some elliptic point
   near the geodesic line (we can control the number of elliptic points
   by Lemma 2.10 in \cite{Iwa95}), we have
 \begin{align}\label{nearellipticavr}
   \int_{\Re(z)=x\atop y>0}R(z,z)\psi(y)\frac{dy}{y}=2\int_{0}^{\infty}\psi(y)\frac{dy}{y}+O\left( Y^2 k^{-\frac{1}{2}}(\log k)^{\frac{1}{2}} \right).
 \end{align}
where the error term is by an observation that a hyperbolic circle in the upper half-plane with hyperbolic radius $r$ corresponds to an Euclidean circle with Euclidean radius $\sinh(r)$ in the same position.

 The proof for the horizontal case is similar to that for the vertical geodesic case.
 So we omit the details.
 We can show that for any $y\gg k^{-1/4}(\log k)^{\frac{1}{4}}$ and
 $\phi$ an integrable function on $(-1/2,1/2)$,
 the following asymptotic formula holds uniformly:
 \begin{align}
   \int_{\Im(z)=y \atop |x|<\frac{1}{2}}R(z,z)\psi(x)dx=2\int_{-\frac{1}{2}}^{\frac{1}{2}} \psi(x)dx+O\left( Y^2 k^{-\frac{1}{2}}(\log k)^{\frac{1}{2}} \right).
 \end{align}

\bibliographystyle{amsplain}

\end{document}